\documentclass[11pt]{amsart}

\usepackage{mathrsfs,graphicx,latexsym,tikz,color,euscript}
\usepackage{amsfonts,amssymb,amsmath,amscd,amsthm}
\usepackage{epstopdf}
\usepackage{hyperref}
\usepackage{upgreek}
\usepackage{comment}
\usepackage{pdfsync}

\usepackage{marginnote}
\newcommand{\rr}{\mathbb R}
\newcommand{\cc}{\mathbb C}
\newcommand{\zz}{\mathbb Z}

\newcommand{\A}{\mathcal A}

\newcommand{\bs}{\mathbb S}

\newcommand{\ee}{\mathbf e}
\newcommand{\I}{\mathbf I}

\newcommand{\al}{\alpha}
\newcommand{\bt}{\beta}
\newcommand{\gm}{\gamma}
\newcommand{\ep}{\varepsilon}

\newcommand{\dl}{\delta} 
\newcommand{\tht}{\theta}

\newcommand{\lmd}{\lambda}

\newcommand{\sg}{\sigma}
\newcommand{\kp}{\kappa}

\newcommand{\Lmd}{\Lambda}

\newcommand{\ey}{\frac{1}{2}} 
\newcommand{\pwa}{\frac{2}{2+\al}}
\newcommand{\pwb}{\frac{2 +\al}{2}}

\newcommand{\xb}{\bar{x}}

\newcommand{\tnb}{\tilde{\nabla}}

\def\me#1{\langle\!\langle #1 \rangle\!\rangle}

\theoremstyle{plain}
\newtheorem{thm}{Theorem}[section]
\newtheorem{prop}{Proposition}[section]
\newtheorem{lem}{Lemma}[section]

\theoremstyle{definition}
\newtheorem{dfn}{Definition}[section]

\theoremstyle{remark}
\newtheorem{rem}{Remark}[section]

\begin{document}

\title[positive energy solutions]{Positive energy solutions in the anisotropic Kepler problem with homogeneous potential}
 

\author{Guowei Yu}
\address{Chern Institute of Mathematics and LPMC, Nankai University, Tianjin, China}
\email{yugw@nankai.edu.cn}

\thanks{This work is supported by the National Key R\&D Program of China (2020YFA0713303), NSFC (No. 12171253), Nankai Zhide Foundation and the Fundamental Research Funds for the Central Universities.}

\begin{abstract}

We study positive energy solutions of the anisotropic Kepler problem with homogeneous potential. First some asymptotic property of positive energy solutions is obtained, as time goes to infinity. Afterwards, we prove the existence of hyperbolic solutions with given initial configuration and asymptotic behavior, when time goes to positive or negative infinity, and in the planar case, the existence of bi-hyperbolic solutions with given asymptotic behaviors, when time goes to both positive and negative infinities, under various conditions.  
\end{abstract}

\maketitle

\section{Introduction }\label{sec:intro}
In this paper, we study the following second order differential equation in $\rr^d$ with $d \ge 2$  
\begin{equation}
\label{eq;aniso-kepler} \ddot{x} = \nabla U(x),  
\end{equation}
under the following conditions

\begin{equation}
\label{eq;U-homo}  \begin{cases}
U(x) = |x|^{-\al}U( \frac{x}{|x|}), \; & \text{ with } \al \in (0, 2); \\
U|_{\bs} \in C^2(\bs,\; \rr^+), \; & \text{ with } \bs := \{ x \in \rr^d: |x| =1\}. 
\end{cases} \tag{U0}
\end{equation}
Recall that the energy of the system, given as below, is constant along a solution  
\begin{equation}
\label{eq;energy} h = \ey |\dot{x}|^2 - U(x).
\end{equation}

When $\al=1$ and $U|_{\bs}$ is a constant, it is the well-known Kepler problem, which is an integrable system. Moreover, we have explicit expressions for all solutions, which are classified as elliptic, parabolic and hyperbolic solutions corresponding to negative, zero and positive energies. 

In this paper, we are mainly interested in the case that $U|_{\bs}$ is not a constant, which we refer as \emph{the anisotropic Kepler problem (with $\al$-homogeneous potential)}. The cases we are mostly interested in are 
\begin{equation} \tag{U1}
\label{eq;U-aniso} U(x) = \langle x, M x \rangle^{-\frac{\al}{2}}, \text{ where } M = \text{diag}(m_1, \dots, m_d) \text{ with each } m_i >0. 
\end{equation}

When $\al=1$ and $m_i$ are not identical to each other, it was first investigated by the physicist Martin Gutzwiller, as it can describe the motion of an electron inside a semiconductor with a donor impurity. This problem also played an important role in the development of quantum chaos. 

In \cite{Gutz71} Gutzwiller found the famous \emph{Gutzwiller trace formula} and applied it to this problem, when $d=2$ and conditions \eqref{eq;U-aniso} and \eqref{U;spiral} (given below) are satisfied. Later under the same conditions, in \cite{Gutz73} and \cite{Gutz77}, he further investigated periodic solutions and chaotic behaviors in this problem. Rigorous proofs were given by Devaney in \cite{Dev78a} using McGehee's blow-up method developed in the study of triple collision for the three-body problem \cite{Mg74}.

For more background, we refer the reader to Gutzwiller's book \cite{Gutz90}. To distinguish it from the general case, when condition \eqref{eq;U-aniso} holds, we call it \emph{Gutzwiller's anisotropic Kepler problem} (with homogeneous $\al$-potential).

The homogeneous condition in \eqref{eq;U-homo} implies the energy of bounded solutions must be negative and solutions with non-negative energies must be unbounded. Zero energy solutions have recently been studied by Barutello, Terracini and Verzini in \cite{BTV13}  and \cite{BTV14}  using McGehee's blow-up method and variational methods. In \cite{HY18}, Hu and the author studied the Morse indices of zero energy solutions of the planar anisotropic Kepler problem using McGehee's blow-up method and the Maslov-type index theory \cite{Long02}. More recently, Barutello, Canneori and Terracini \cite{BCT21} explored symbolic dynamics of the planar $N$-center problem with anisotropic forces, when energies are negative but close to zero.

 Meanwhile for positive energy solutions, very few results seem to be available, and it will be the main focus of our paper. 


To explain our results, first let's rewrite \eqref{eq;aniso-kepler} in polar coordinates: $(r, s) = (|x|, \frac{x}{|x|}) \in [0, \infty) \times \bs,$ 
\begin{equation}
\label{eq;aniso-keper-polar} \begin{cases}
& \ddot{r} = r|\dot{s}|^2 + r^{-(\al +1)} \langle \nabla U(s), s \rangle; \\
& r \ddot{s} + 2\dot{r} \dot{s} + r|\dot{s}|^2 s = r^{-(\al +1)} \tnb U(s).
\end{cases} 
\end{equation}
Here $\tnb U(s)$ represents the gradient of $U$ restricted to $\bs$  
\begin{equation}
\label{eq;til-nb} \tnb U(s) = \nabla U(s) - \langle \nabla U(s), s \rangle s. 
\end{equation}
Denote the set of critical points of $U$ on $\bs$ by
$$  Cr(U): = \{s \in \bs: \; \tilde{\nabla} U(s)=0 \}. $$ 
The energy identity in polar coordinates can be written as 
\begin{equation}
\label{eq;energy-polar} h = \ey \dot{r}^2 + \ey r^2 |\dot{s}|^2 - r^{-\al} U(s). 
\end{equation}

The problem has a singularity at the origin, where the potential is infinite. When a solution approaches the origin, it satisfies the following asymptotic estimates.  
\begin{prop}{\cite[Section 4]{BFT08}}
\label{prop;Sundman-est} Given a solution $x(t) =r(t) s(t)$ with $r(t) \to 0$, when $t \to t_0^{\pm}$, then there exist positive constants $\kp_{\pm}$, such that the following asymptotic results hold, as $t \to t_0^{\pm}$,
\begin{enumerate}
\item[(a).] $r(t) \simeq (\kp_{\pm} |t-t_0|)^{\frac{2}{2+\al}}$, $\dot{r}(t) \simeq \pm \frac{2\kp_{\pm}}{2 +\al}\big( \kp_{\pm} |t-t_0| \big)^{\frac{-\al}{2+\al}}$;
\item[(b).] $U(s(t)) \to \bt_{\pm} = \ey \left( \frac{2}{2 +\al} \kp_{\pm} \right)^2$;
\item[(c).] $\text{dist}\big(s(t), \text{Cr}(U;\bt_{\pm}) \big) \to 0$, where $\text{Cr}(U; \bt_{\pm}) = \{s \in \text{Cr}(U): U(s) = \bt_{\pm} \}$. 
\end{enumerate}
\end{prop}

The above type of asymptotic estimates are well-known for collision solutions in celestial mechanics (\cite{Wt41}, \cite{Sp70}, \cite{C98} and \cite{FT04}). In \cite{BFT08}, it was generalized to a large class of quasi-homogeneous potentials with singularity including the case we are considering here.  

\begin{thm} \label{thm;asym} Let $x(t)$ be a positive $h$-energy solution of \eqref{eq;aniso-kepler} with maximal domain $(T^-, T^+)$. For $T^+$, it is either a finite number or positive infinity, and
\begin{enumerate}
\item[(a).] if $T^+$ is finite, then $x(t) \to 0$, when $t \to T^{+}$;
\item[(b).] if $T^{+} = + \infty$, then there exists an $s^{+} \in \bs$, such that when $t \to +\infty$, 
\end{enumerate}
\begin{equation} \label{eq;asym-infty}
 x(t) =
\begin{cases}
\sqrt{2h} t s^{+}-\frac{\nabla U(s^{+})}{2h} \log t+ o(\log t), \; & \text{ if } \al =1;\\
\sqrt{2h} t s^+ -\frac{\nabla U(s^+)}{\al(1-\al)(\sqrt{2h})^{\al+1}} t^{1-\al} + o(t^{1-\al}), \; & \text{ if } \al \in (0, 2) \setminus \{1\}.
\end{cases}
\end{equation} 
A similar result as above holds for $T^-$.
\end{thm}

\begin{rem}
	Although the second order term is obtained in \eqref{eq;asym-infty}, for the sake of simplicity in the theorems stated below we will only write the first order term in the corresponding asymptotic expansions. 
\end{rem}

By the above result, a positive energy solution either approaches the singularity at the origin at a finite time (forwardly or backwardly), or becomes unbounded as time goes to positive or negative infinity. Moreover in the latter it must satisfy the asymptotic expression given in \eqref{eq;asym-infty}. 
 
By Lemma \ref{lem;Lag-Jacobi}, $T^{\pm}$ in the above theorem cannot both be finite. Following notations from celestial mechanics, we will also refer positive energy solutions as \textbf{hyperbolic solutions}. In particular, a positive energy solution will be called \textbf{bi-hyperbolic}, if its maximal domain is $\rr$.

With the above result, we are interested in the following two questions: 

\begin{enumerate}
\item[(Q1):] \emph{for any $h>0$, starting from a given initial configuration, is there an $h$-energy solution with the prescribed asymptotic expression as time goes to negative or positive infinity?}

\item[(Q2):] \emph{for any $h>0$, is there an $h$-energy solution with given asymptotic expressions, when time approaches both positive and negative infinities?}
\end{enumerate}

For the $n$-body problem, a question similar to (Q1) was completely answered first in \cite{MV20}, and later in \cite{LYZ21}, \cite{Yu24} and \cite{PT24} with different approaches but all variational in nature. Meanwhile a question similar to (Q2) is far from being answered, up to our knowledge only partial results are available in \cite{DMMY20}. 

As we can see, (Q2) is closely related to scattering theory, and since Rutherford, the major part of our knowledge about molecules, atoms and elementary particles is obtained by scattering experiment. For more details see \cite[Chapter 12]{Kn18}.

Our approach in this paper is variational as well. Recall that solutions of \eqref{eq;aniso-kepler} correspond to critical points of the following action functional 
\begin{equation}
\A(\gm; t_1, t_2) = \int_{t_1}^{t_2} L(\gm, \dot{\gm}) \,dt = \int_{t_1}^{t_2} \ey |\dot{\gm}|^2 + U(\gm) \,dt, \; \gm \in H^1_{t_1, t_2}(p, q).
\end{equation}
For any $p, q \in \rr^d$ and $t_1< t_2 \in \rr$, $H^1_{t_1, t_2}(p, q)$ denotes the set of all Sobolev paths going from $p$ to $q$, as time goes from $t_1$ to $t_2$.

\begin{dfn}
\label{dfn;fixed-time minimizer}  We say $\gm \in H^1_{t_1, t_2}(p, q)$ is \textbf{a fixed-time minimizer} of the action functional $\A$, if 
$$ \A(\gm; t_1, t_2) \le \A(\xi; t_1, t_2), \; \forall \xi \in H^1_{t_1, t_2}(p, q),$$
and \textbf{a fixed-time local minimizer}, if the above $\xi$'s further satisfy
$$ \|\xi -\gm\|_{H^1} < \dl, \text{ for some } \dl>0 \text{ small enough.}$$
\end{dfn}

A minimizer, or more general a critical point of $\A$, is a classical solution of \eqref{eq;aniso-kepler}, only if it does not contain any singularity. We call such a minimizer \emph{collision-free}, following notations from celestial mechanics. To apply variational methods to our problem, or any problem with singularity, it is crucial to rule out singularities in the corresponding critical points.

When $\al \ge 2$, it is well-known any critical point of $\A$ with finite action value can not contain any singularity. However for $\al \in (0, 2)$, critical points with finite action value could still contain singularities. In particular, if only condition \eqref{eq;U-homo} is satisfied, a local minimizer may not always be collision-free. 

For the $n$-body problem, a local minimizer must be collision-free. This was proven in \cite{C02} for the Newtonian potential and in \cite{FT04} for $\al \in (0, 2)$ using Marchal's averaging method \cite{Mc02}. Such a method can be applied to the anisotropic problem, if the potential is invariant under certain rotational symmetry, and the following result was proven in \cite{BFT08}. 

\begin{prop}{\cite[Theorem 6.1]{BFT08}} \label{prop;coll-free-1}
	A fixed-time local minimizer of $\A$ is collision-free, if
\begin{equation}
\label{eq;U-inv} \text{there exists a two-dim linear subspace } W \subset \rr^d \text{ with } U|_{W \cap \bs} = \text{constant}. \tag{U2}
\end{equation}
\end{prop}

Here we prove that for Gutzwiller's anisotropic Kepler problem, a local minimizer must be collision-free. 

\begin{prop} \label{prop;coll-free-2}
Under condition \eqref{eq;U-aniso}, a fixed-time local minimizer of $\A$ is collision-free.  
\end{prop}

\begin{rem}
Notice that for Gutzwiller's anisotropic Kepler problem, condition \eqref{eq;U-inv} is satisfied, if $m_i = m_j$ for some $i \ne j$, but not if $m_i \ne m_j$ for any $i \ne j$. In particular, for the planar case, condition \eqref{eq;U-inv} is satisfied only when it is the Kepler problem.  
\end{rem}

Using the above propositions, we can prove the following result regarding (Q1).

\begin{thm}
\label{thm;hyper} Under condition \eqref{eq;U-aniso} or \eqref{eq;U-inv}, for any $h>0$, $x_0 \in \rr^d$ and $s^{\pm} \in \bs$, there is an $h$-energy hyperbolic solution $x: [t_0, \pm \infty) \to \rr^d$ satisfying $x(t_0) =x_0$ and 
$$ x(t) = \sqrt{2h} |t| s^{\pm} + o(|t|), \; \text{as } t \to \pm \infty. $$  
\end{thm}

\begin{rem}
	When $x_0=0$, by $x(t_0)=x_0$, we mean  $\lim_{t \to t_0^+} x(t) =0$.
\end{rem}

\begin{rem}\label{rem:coll-free} From our proof of the theorem in Section \ref{sec;hyp-sol}, one can see that the above result holds as long as the fixed-time minimizers of $\A$ are collision-free.
\end{rem}
For (Q2), the main difficulty is the lack of coercivity, when we approach it with variational methods. In the planar case, one way to force coercivity is to impose appropriate topological constraint. The problem then is minimizers are not always collision-free as shown in \cite{Gordon77} and \cite{BTV13}. Despite of this, we are able to give some partial answer to (Q2). To explain it, set $\rr^2 \simeq \cc$ and define 
$$ \tilde{U}: \rr \to \rr; \; \tht \mapsto U(e^{i\tht}).$$
Denote the set of minimizers of $\tilde{U}$ as 
\begin{equation} \label{eq;min-U}
\mathcal{M}(\tilde{U}) :=\{ \tht_0 \in \rr: \tilde{U}(\tht_0) \le \tilde{U}(\tht), \; \forall \tht \in \rr \}.
\end{equation}

\begin{thm}
\label{thm;bi-hyper} When $d=2$, under the non-degenerate condition  
\begin{equation} \tag{U3}
\label{eq;U-nondeg}  \tilde{U}''(\tht_0) \ne 0,  \; \forall \tht_0 \in \mathcal{M}(\tilde{U}),
\end{equation}
for any $\tht_{\pm} \in \rr$ satisfying $|\tht_+-\tht_-| >\pi$, there is a constant $\bar{\al} =\bar{\al}(U, \tht_{\pm}) \in (0, 2)$, such that for any $\al \in (\bar{\al}, 2)$, there is an $h$-energy bi-hyperbolic solution $x: \rr \to \cc$ satisfying
$$x(t) = \sqrt{2h} |t| e^{i \tht_{\pm}} + o(|t|), \; \text{ as } t \to \pm \infty.$$
\end{thm}
\begin{rem}
The constant $\bar{\al}$ was found in \cite{BTV13}, for the detail see Proposition \ref{prop;BTV}. 
\end{rem}
 
While condition \eqref{eq;U-nondeg} is relatively weak, the constant $\bar{\al}$ could be larger than 1, which rules out the most important Newtonian potential. On the other hand, for Gutzwiller's anisotropic Kepler problem, we are able to obtain a result, which holds for all $\al \in (0, 2)$.   

\begin{thm}
\label{thm;bi-hyperbolic-aniso} When $d=2$, under the conditions \eqref{eq;U-aniso} and  
\begin{equation} \label{U;spiral}  \frac{m_2}{m_1} > 1 + \frac{(2-\al)^2}{8\al}
\tag{U4} 
\end{equation} 
for any pair of $s^{\pm} = s^{\pm}_1+ is^{\pm}_2 \in \cc$ with $|s^{\pm}|=1$, if one of the following conditions holds, 
\begin{enumerate}
	\item[(i).] $s^+ \ne -s^-$, $s^-_1 \ne \pm 1$ and $s^-_2 s^+_2 \le 0$;
	\item[(ii).] $s^+ \ne -s^-$, $s^-_1 = \pm 1$, 
\end{enumerate}
there exists at least one $h$-energy bi-hyperbolic solution $x: \rr \to \cc$ satisfying
$$x(t) = \sqrt{2h} |t| s^{\pm} + o(|t|), \; \text{as } t \to \pm \infty.$$

\end{thm}

\section{Asymptotic analysis of positive energy solutions}

Throughout this section, let $x(t)= r(t)s(t)$ be a positive $h$-energy solution with maximal domain $(T^-, T^+)$. We introduce two auxiliary functions that will be needed in our proofs.  
\begin{equation}
\label{eq;I-Gamma} I(t) = I(x(t)) = r^2(t); \; \; \Gamma(t) = \Gamma (x(t)) = \frac{1}{2}r^{\al}(t)(2h - \dot{r}^2(t)). 
\end{equation} 
\begin{lem}[Lagrange-Jacobi identity]
\label{lem;Lag-Jacobi} $ \forall t \in (T^-, T^+)$, $\ddot{I}(t) = 4h + 2(2-\al)U(x(t)) >4h$.
\end{lem}

\begin{proof}
By a direct computation, $\dot{I} = 2 \langle \dot{x}, x \rangle$. Then for $\al \in (0, 2)$,
$$
\begin{aligned}
 \ddot{I}/2 & = |\dot{x}|^2 + \langle x, \ddot{x} \rangle = |\dot{x}|^2 + \langle x, \nabla U(x) \rangle \\
 			  & = 2 ( h + U(x)) - \al U(x) = 2h + (2-\al) U(x) >2h.
 \end{aligned} 
$$
\end{proof}

\begin{lem}
\label{lem;I-ge} For any $T^-<t_0< t_1 <T^+$, 
\begin{equation}
\label{eq;I-ge} I(t_1) \ge 2h(t_1 -t_0)^2 + \dot{I}(t_0)(t_1 -t_0) + I(t_0).  
\end{equation}
\end{lem}
\begin{proof}
By the above lemma, $\ddot{I} \ge 4h$, as a result for any $t_0 \le t \le t_1$, 
\begin{equation}
\label{eq;I-dot} \dot{I}(t) \ge 4ht - 4h t_0 + \dot{I}(t_0).
\end{equation}
Integrating both sides of the above inequality from $t_0$ to $t_1$ gives us the desired result. 
\end{proof}

\begin{lem}
\label{lem;Gamma-dot} For any $t \in (T^-, T^+)$, $\dot{\Gamma}(t) = -\frac{2 -\al}{2} r^{\al +1} \dot{r} |\dot{s}|^2$.
\end{lem}
\begin{proof}
By the energy identity \eqref{eq;energy-polar}, 
\begin{equation}
\label{eq;Gamma} \Gamma(t) = \ey r^{2 +\al} |\dot{s}|^2 - U(s) 
\end{equation}
Then a direct computation shows
$$ \dot{\Gamma} = \frac{2 +\al}{2} r^{\al +1} \dot{r} |\dot{s}|^2 + r^{2+\al} \langle \dot{s}, \ddot{s} \rangle - \langle \nabla U(s), \dot{s} \rangle. $$
Meanwhile by taking an inner product of the second equation in \eqref{eq;aniso-keper-polar} with $\dot{s}$, we get 
$$ r^{2 +\al} \langle \dot{s}, \ddot{s} \rangle = \langle \nabla U(s), s \rangle - 2 r^{\al +1} \dot{r} |\dot{s}|^2. $$
Plugging this into the previous equation gives us the desired result. 
\end{proof}

\begin{lem}
\label{lem;|s-dot|} If $\dot{I}(t_0) \ge 0$ for some $t_0 \in (T^-, T^+)$, then for any $t_0 < t_1 < t_2 < T^+$, 
\begin{equation}
\label{eq;int-s-dot} \int_{t_1}^{t_2} |\dot{s}| \,dt \le  \frac{2}{\al} \frac{\sqrt{2h r^{\al}(t_0)+2U_{max}}}{(2h )^{\frac{2 +\al}{4}}}(t_1-t_0)^{-\frac{\al}{2}} .
\end{equation}
\end{lem}

\begin{proof}
Notice that $\dot{r}(t) >0$, $\forall t >t_0$, as the assumption and \eqref{eq;I-dot} imply
$$ \dot{I}(t) = 2r(t) \dot{r}(t) \ge 4h(t-t_0) + \dot{I}(t_0) >0, \; \forall t >t_0. $$
Then Lemma \ref{lem;Gamma-dot} implies $\dot{\Gamma}(t) \le 0$, $\forall t >t_0$. As a result, when $r_0 = r(t_0)$,
$$ \Gamma(t) \le \Gamma(t_0) \le h r^{\al}_0, \; \forall t > t_0. $$
Let $C_1 = 2U_{max}$, then the above inequality and \eqref{eq;Gamma} imply
$$ r^{2 +\al}(t) |\dot{s}(t)|^2 \le 2\big(h r_0^{\al} + U(s(t)) \big) \le 2h r_0^{\al} + C_1. $$
Meanwhile according to \eqref{eq;I-ge}, 
$$ I(t) =r^2(t) \ge 2h(t-t_0)^2, \; \forall t >t_0. $$
Combining these estimates, we get 
$$ |\dot{s}(t)| \le \left( \frac{2h r_0^{\al}+C_1}{r^{2 +\al}(t)} \right)^{\ey} \le \frac{\sqrt{2h r_0^{\al}+C_1}}{(2h(t-t_0)^2)^{\frac{2+\al}{4}}} \le  \frac{\sqrt{2h r_0^{\al}+C_1}}{(2h )^{\frac{2+\al}{4}}} (t-t_0)^{-1 -\frac{\al}{2}}. $$
The result now follows from a direct integration. 
\end{proof}

With the above lemmas, we can prove the following asymptotic estimates. 
\begin{prop} \label{prop;r-s-lim} 
When $T^{\pm} = \pm \infty$, the following asymptotic estimates hold. 
\begin{enumerate}
\item[(a).] $r(t) = \sqrt{2h} |t| + o(|t|)$, as $t \to \pm \infty$.
\item[(b).] there is a $s^{\pm} \in \bs$, such that $\lim_{t \to \pm \infty} s(t) = s^{\pm}$. 
\end{enumerate}
\end{prop}

\begin{proof}
We only give the details for $T^+=+\infty$.

(a). By Lemma \ref{lem;Lag-Jacobi} and \ref{lem;I-ge}, $r^2(t) = I(t) \to \infty$ and
\begin{equation}
\label{eq;ddot-dot-I} \lim_{t \to \infty} \frac{I}{t^2} = \lim_{t \to \infty} \frac{\dot{I}}{2t} = \lim_{t \to \infty} \frac{\ddot{I}}{2} = 2h+  \lim_{t \to \infty} 2(2-\al)\frac{U(s)}{r^{\al}} = 2h,
\end{equation} 

As $\dot{I} = 2 r \dot{r}$, plugging this into the above identities, we get 
$$2h = \lim_{t \to \infty} \frac{r \dot{r}}{t} = \lim_{t \to \infty} \left(  t^{-2} I\right)^{\ey} \dot{r} = \sqrt{2h} \lim_{ t \to \infty} \dot{r}. $$
Since $\lim_{t \to \infty} \sqrt{I/t^2} = \sqrt{2h}$, we get 
$$ \lim_{t \to \infty} \dot{r} = \lim_{t \to \infty} \frac{r}{t} = \sqrt{2h}. $$
This gives us the desired property. 

(b). By Lemma \ref{lem;Gamma-dot}, $\Gamma(t)$ is non-increasing for $t$ large enough, this implies   
\begin{equation}
\label{eq;Gm-lim} \lim_{t \to \infty} \Gamma(t) = \lim_{t \to \infty} \ey r^{2 +\al}|\dot{s}|^2 - U(s) = C_1,
\end{equation}
and $\ey r^{2 +\al} |\dot{s}|^2 \le C_2$, for some constant $C_2 >0$. Then for $t$ large enough
\begin{equation}
\label{eq;s-dot-le} |\dot{s}(t)| \le \sqrt{2 C_2} r(t)^{-\frac{2 + \al}{2}} \le C_3 t^{-\frac{2 +\al}{2}}.
\end{equation}
This means for $t_1 >t_0$ large enough, 
$$ \lim_{t_2 \to \infty} \int_{t_1}^{t_2} |\dot s(t) | \,dt \le \lim_{t_2 \to \infty} \frac{\al}{2} C_2 (t_1^{-\frac{\al}{2}}- t_2^{-\frac{\al}{2}}) \le \frac{\al}{2} C_2 t_1^{-\frac{\al}{2}} < \infty,
$$
which implies the convergence of $s(t)$, when $t \to \infty$. 
\end{proof}

Theorem \ref{thm;asym} now follows directly from the above results.  
 
\begin{proof}[Proof of Theorem \ref{thm;asym}]
We will only give the proof for $T^+$.

(a). By the existence theorem of ordinary differential equations, if $T^{+}$ is finite, when $t \to T^{+}$, $x(t)$ must approach to the set of singularities of $U$, which in our case is a single point corresponding to the origin. 

(b). First by Proposition \ref{prop;r-s-lim},  
$$ \lim_{t \to +\infty} \frac{x(t)}{t} = \lim_{t \to +\infty} \frac{r(t)}{t} s(t) = \sqrt{2h} s^{+}. 
$$
This means
$$ x(t) = \sqrt{2h}t s^+ +o(t), \text{ as } t \to +\infty$$
and 
$$ \dot{x}(t) = \sqrt{2h}s^+ +o(1), \text{ as } t \to \infty. $$
The higher order terms depend on the value of $\al$, and we divide the proof into two cases.

\emph{Case 1:} $\al=1$.  By the above estimates, 
$$ \begin{aligned}
\lim_{t \to +\infty} \frac{\dot{x}(t)-\sqrt{2h}s^+}{t^{-1}} & = \lim_{t \to +\infty} -\frac{\ddot{x}(t)}{t^{-2}} =  \lim_{t \to +\infty} -\frac{\nabla U(x(t))}{t^{-2}}\\
& =  -\lim_{t \to +\infty} \frac{\nabla U(\sqrt{2h}t s^+ + o(t))}{t^{-2}} = -\frac{1}{2h} \nabla U(s^+). 	
\end{aligned}
$$
This implies that, when $t \to +\infty$,
$$ \dot{x}(t) = \sqrt{2h} s^+ - \frac{1}{2h} \nabla U(s^+) t^{-1} + o(t^{-1}).
$$
The desired result now follows from an integration of the above equation.  

\emph{Case 2:}  $\al \in (0, 2) \setminus \{1\}$. The proof is similar to the previous case. There only difference is 
$$ \begin{aligned}
\lim_{t \to +\infty} \frac{\dot{x}(t)-\sqrt{2h}s^+}{t^{-\al}} &= \lim_{t \to +\infty} -\frac{\ddot{x}(t)}{\al t^{-\al-1}} = \lim_{t \to +\infty} -\frac{\nabla U(x(t))}{\al t^{-\al-1}}\\
& =  -\lim_{t \to +\infty} \frac{\nabla U(\sqrt{2h}t s^+ + o(t))}{\al t^{-\al-1}} = -\frac{1}{\al (2h)^{\frac{\al+1}{2}}} \nabla U(s^+),
\end{aligned}
$$
and this implies that, when $t \to +\infty$, 
$$  \dot{x}(t) = \sqrt{2h} s^+ -  \frac{1}{\al (2h)^{\frac{\al+1}{2}}} \nabla U(s^+) t^{-\al} + o(t^{-\al}).
$$
\end{proof}



\section{Hyperbolic solutions with given initial configuration} \label{sec;hyp-sol}

A proof of Theorem \ref{thm;hyper} will be given in this section. To obtain solutions with prescribed energy, instead of fixed-time minimizers, we will look for free-time minimizers defined as below. 

\begin{dfn}
\label{dfn;free-time-minimizer} For any $h>0$, define \textbf{the action potential} $\phi_h: \rr^d \times \rr^d \to \rr$ as 
$$ \phi_h(p, q) = \inf \left\{ \A_h(\xi; \tau_1, \tau_2) := \int_{t_1}^{t_2} L(\xi, \dot{\xi}) +h \,dt : \; \xi \in \cup_{-\infty < \tau_1 < \tau_2 < \infty} H^1_{\tau_1, \tau_2}(p, q) \right\}.$$
We say $\gm \in H^1_{t_1, t_2}(p, q)$ is \textbf{an $h$-free-time minimizer}, if 
$$ \A_h(\gm; t_1, t_2) = \int_{t_1}^{t_2} L(\gm, \dot{\gm}) +h \,dt = \phi_h(p, q).$$
\end{dfn}

\begin{lem}
\label{lem;exist-free-minimizer} For any $h>0$ and $p \ne q \in \rr^d$, there exists at least one $\gm \in H^1_{0, T}(p, q)$ satisfying 
$$ \A_h(\gm; 0, T) = \phi_h(p, q). $$
Moreover if $\gm$ is collision-free, it is an $h$-energy solution of \eqref{eq;aniso-kepler}.
\end{lem}
\begin{proof}
First let's show the existence of such a $\gm$. For this, define a function $f: (0, \infty) \to \rr$ as 
$$  f(T) = \inf \{ \A_h(\xi; 0, T): \; \xi \in H^1_{0, T}(p, q) \}. $$
First we claim the above infimum is in fact a minimum. As $\A_h$ is lower semi-continuous, we just need to show it is coercive in $H^1_{0, T}(p, q)$, i.e., $\A_h(\xi; 0, T) \to \infty$, if $\|\xi\|_{H^1} \to \infty$. 

By the Cauchy-Schwarz inequality, for any $t \in [0, T]$
\begin{equation}
\label{eq;gm-L2} |\xi(t)| \le |\xi(t)-p| + |p| \le \|\xi\|_{L^1} + |p| \le \sqrt{T}\|\dot{\xi}\|_{L^2} + |p|.
\end{equation}
Hence 
$$ \|\xi\|_{L^2}^2 \le (\sqrt{T}\|\dot{\xi}\|_{L^2} + |p|)^2 T. $$
As a result, $\|\dot{\xi}\|_{L^2} \to \infty$, if $\|\xi\|_{H^1} \to \infty$. This implies coercivity, as $\A_h(\xi; 0, T) \ge \ey \|\dot{\xi}\|^2_{L^2}$.

Then for any $T>0$, we can find a $\gm_T \in H^1_{0, T}(p, q)$ with $\A_h(\gm_T; 0, T) = f(T)$. By \eqref{eq;gm-L2},   
$$f(T)= \A_h(\gm_T;0, T) \ge \ey \|\dot{\gm}_T\|^2_{L^2} \ge \frac{|p -q|^2}{2\sqrt{T}} \to \infty, \text{ if } T \to 0. $$
Meanwhile $f(T) \to \infty$, if $T \to \infty$, as $\A_h(\gm_T; 0, T) \ge hT$. 

As a result, it is enough to show $f(T)$ is continuous. For any $\dl>0$, define
$$ \gm^{\dl}_T(t) = \gm_T(t/\dl), \; \forall t \in [0, \dl T]. $$
Then 
$$  \begin{aligned}
f(\dl T) & \le \A_h(\gm^{\dl}_T; 0, \dl T) = \int_0^T \frac{1}{2 \dl} |\dot{\gm}_T|^2 + \dl (U(\gm_T) + h) \,dt \\
& = \dl \A_h(\gm_T; 0, T) + (\dl^{-1}-\dl) \int_0^T \ey |\dot{\gm}_T|^2 \,dt \le \dl^{-1} f(T). 
\end{aligned}
$$
Meanwhile by a similar argument, $f(T) \le \dl f(\dl T)$. As a result, 
$$ |f(\dl T) -f(T)| \le \dl^{-1}|1-\dl| f(T) \to 0, \text{ as } \dl \to 1. $$
This proves the continuity of $f(T)$.

When $\gm$ is a collision-free, it is a smooth solution of \eqref{eq;aniso-kepler}. Then being an $h$-free-time minimizer implies its energy must be $h$, it is a standard result from Aubry-Mather theory and a detailed proof can be found in \cite[Proposition 4.1.23]{So15}
\end{proof}


We give some estimates of the action potential that will be needed later. 

\begin{lem}
\label{lem;phh-le} Let $U_{max}:= \max\{ U(s): \; s \in \bs\}$.
\begin{enumerate}
\item[(a).] For any $s \in \bs$ and $r_2 >r_1 >0$, 
$$ \phi_h(r_1 s, r_2 s) \le \sqrt{2h}(r_2 -r_1)+ \begin{cases}
\frac{U_{max}}{\sqrt{2h}}(\log r_2 - \log r_1) , \; & \text{ when } \al =1; \\ 
\frac{U_{max}}{\sqrt{2h}(1-\al)} (r_2^{1-\al} - r_1^{1-\al}), \; & \text{ when } \al \in (0, 1) \cup (1, 2).
\end{cases}
$$ 
\item[(b).] For any $x \ne 0$, there is a positive constant $C_1$ depending on $x$, such that 
$$ \sup \{ \phi_h(x, |x|s): \; s \in \bs \} \le C_1. $$
\item[(c).] There is a constant $C_2>0$, such that 
$$ \sup \{\phi_h(0, s): \; s \in \bs \} \le C_2.$$
\end{enumerate}

\end{lem}
\begin{proof}
(a). Define a path $\xi$ as below. The result then follows from a direct computation 
$$  \xi(t) = (r_1+ \sqrt{2h}t)s, \; \forall t \in \left[0, \frac{r_2 -r_1}{\sqrt{2h}} \right]$$. 

(b). For any $x$ and $|x|s$, one can simply choose the shortest geodesic on the sphere in $\rr^d$ with radius $|x|$. 

(c). Define a path $\xi$ as below. The result again follows from a direct computation
$$ \xi(t) = t^\frac{2}{2 +\al} s, \; \forall t \in [0, 1]. $$

\end{proof}

\begin{proof}[Proof of Theorem \ref{thm;hyper}]
We only give the details for $t \to +\infty$, while the other is the same. 

First let's consider the case that $r_0 = |x_0|>0$. Choose a sequence of numbers $\{ R_n \nearrow \infty\}_{n=1}^{\infty}$ with each $R_n > r_0$. By Lemma \ref{lem;exist-free-minimizer}, for each $n$, there is a $\gm_n \in H^1_{0, T_n}(x_0, R_n s^+)$ satisfying 
$$ \A_h(\gm_n;0, T_n) = \phi_h(x_0, R_n s^+). $$ 

We claim $T_n \to \infty$, when $n \to \infty$. To show this, by Proposition \ref{prop;coll-free-1} and \ref{prop;coll-free-2}, each $\gm_n$ is collision-free, and hence an $h$-energy solution. Then for any $T \in (0, T_n]$,  
$$ \int_0^{T} |\dot{\gm}_n|^2 \,dt = \int_{0}^{T} \ey|\dot{\gm}_n|^2 + U(\gm_n)+h \,dt = \phi_h(x_0, \gm_n(T)). $$
By the Cauchy-Schwarz inequality, 
\begin{equation}
\label{eq;xnT-le-phh} |\gm_n(T)-\gm_n(0)|^2 \le \left( \int_0^{T} |\dot{\gm}_n| \,dt \right)^2 \le T \int_{0}^{T} |\dot{\gm}_n|^2 \,dt = T \phi_h(x_0, \gm_n(T)). 
\end{equation}

By property (a) and (b) in Lemma \ref{lem;phh-le}, we can find positive constants $C_1$ and $C_2$ independent of $n$ and $T$, such that  
\begin{equation}
 \label{eq;phh-le-C1} \phi_h(x_0, \gm_n(T)) \le \phi_h \left(x_0, r_0 \frac{\gm_n(T)}{|\gm_n(T)|}\right) + \phi_h \left(r_0 \frac{\gm_n(T)}{|\gm_n(T)|}, \gm_n(T) \right) \le C_1 + C_2 |\gm_n(T)|
 \end{equation} 
By the above two estimates, $|\gm_n(T)|$ satisfies the following inequality, 
\begin{equation}
\label{eq;xnT-le}
 (|\gm_n(T)| - r_0)^2 \le |\gm_n(T) -\gm_n(0)|^2 \le (C_1+ C_2 |\gm_n(T)|)T.
 \end{equation}
This proves our claim, as when $T=T_n$, the above inequality implies
$$ T_n \ge \frac{(R_n - r_0)^2}{C_1 + C_2 R_n} \to \infty, \text{ as } n \to \infty. $$

Meanwhile by solving \eqref{eq;xnT-le} for $|\gm_n(T)|$,  
\begin{equation}
\label{eq;xnT-up-bound} |\gm_n(T)| \le \frac{2r_0 + C_2 T+ \sqrt{4(C_1 + r_0 C_2) T + C_2^2 T^2}}{2}:=C_3(T).
\end{equation}
Combining this with \eqref{eq;phh-le-C1}, we find  
$$ \int_0^T |\dot{\gm}_n|^2 \,dt = \phi_h(\gm_n(0), \gm_n(T)) \le C_1 + C_2|\gm_n(T)| \le C_1 + C_2C_3. $$

This shows $\{\gm_n|_{[0, T]} \}$ is a bounded sequence in $H^1_{0, T}$ and after passing to a proper sub-sequence, it converges to a limiting path weakly in the $H^1$-norm and strongly in the $L^{\infty}$-norm. As this holds for any $T>0$, by a standard diagonal argument, we can show the existence of a path $\gm \in H^1([0, \infty), \rr^d)$, such that after passing to a proper sub-sequence $\gm_n$ converges to $\gm$ weakly in the $H^1_{loc}$-norm and strongly in the $L^{\infty}_{loc}$-norm. 

\begin{lem}
\label{lem;gm-free-minimizer} $\gm$ is an $h$-free-time minimizer, i.e., $\A_h(\gm; 0, T) = \phi_h(\gm(0), \gm(T))$, $\forall T >0$. 
\end{lem}
By the above lemma (we postpone its proof for a moment), $\gm$ is a collision-free $h$-energy solution of \eqref{eq;aniso-kepler}. Then $\gm_n$ converges to $\gm$ in $C^{2}_{loc}$-norm as well. Set 
$$ \begin{aligned}
	r(t)= |\gm(t)|, \; & s(t) = \gm(t)/r(t); \\
	r_n(t)= |\gm_n(t)|,\; & s_n(t) = \gm_n(t) /r_n(t).
\end{aligned} 
$$
Choose an $r_0 >|x_0|$, By Lemma \ref{lem;Lag-Jacobi}, there is a unique $t_0 >0$, such that 
$$ r(t_0)= r_0 \text{ and } \dot{r}(t_0) >0. $$ 
Since $\gm_n \to \gm$ in  $C^{2}_{loc}$-norm, 
$$ r_n(t_0) \le 2 r_0 \text{ and } \dot{r}_n(t_0) >0, \; \text{ for $n$ large enough.}$$
Meanwhile for any $\ep>0$, the following inequality holds, when $t_1>t_0$ is large enough,
$$\frac{2}{\al} \frac{\sqrt{2h r^{\al}_0+2U_{max}}}{(2h )^{\frac{2 +\al}{4}}}(t_1-t_0)^{-\frac{\al}{2}} \le \frac{\ep}{2}.$$
By Lemma \ref{lem;|s-dot|} and the above inequality, 
\begin{equation}
\label{eq;le-ep/2-1} |s_n(t_1)-s^+|=|s_n(t_1)-s_n(T_n)| \le \int_{t_1}^{T_n}|\dot{s}_n(t)|\,dt \le \frac{\ep}{2}.
\end{equation}
Meanwhile when $n$ is large enough, 
$$ |s_n(t_1)-s(t)| \le \ep/2. $$
As a result, 
$$ |s(t_1) - s^+| \le |s(t_1)- s_n(t_1)| + |s_n(t_1)-s^+| \le \ep. $$
Since for any $\ep>0$, this is true for any $t_1$ and $n$ large enough, we have
$$ \lim_{t \to \infty} s(t) = s^+. $$
This finishes our proof for $x_0 \ne 0$. 

The proof of $x_0 =0$ is almost the same, except that \eqref{eq;phh-le-C1} should be replaced by
$$ \phi_h(0, \gm_n(T)) \le \phi_h \left(0, \frac{\gm_n(T)}{|\gm_n(T)|} \right) + \phi_h \left(\frac{\gm_n(T)}{|\gm_n(T)|}, \gm_n(T) \right) \le C_1 + C_2 |\gm_n(T)|. 
$$
Without loss of generality, we may assume $|\gm_n(T)| >1$ and the above inequality can be obtained by property (a) and (c) in Lemma \ref{lem;phh-le}.
\end{proof}

\begin{proof}[Proof of Lemma \ref{lem;gm-free-minimizer}]
By a contradiction argument, assume there is a $T>0$, such that $\gm|_{[0, T]}$ is not an $h$-free-time minimizer, which means we can find an $\ep >0$ and a $\xi \in H^1_S(x_0, \gm(T))$, such that 
\begin{equation*}
\label{eq;3ep} 3\ep = \A_h(\gm; 0, T) -\A_h(\xi; 0, S) = \A_h(\gm; 0, T) - \phi_h(\gm(0), \gm(T)) >0. 
\end{equation*}

For each $n$, let $\dl_n = |\gm(T)- \gm_n(T)|$, we define a path $\xi_n \in H^1_{S +\dl}(x_0, \gm_n(T))$ as below
\begin{equation*}
\label{eq;eta-n} \eta_n(t) = \begin{cases}
\xi(t), & \text{ when } t \in [0, S]; \\
\xi(S) + \frac{t -S}{\dl_n}(\gm_n(T) -\xi(S)), & \text{ when } t \in [S, S+\dl_n]. 
\end{cases}
\end{equation*}
Notice that $\dl_n \to 0$, as $n \to \infty$. Hence we can find a constant $C_1$ independent of $n$, such that for $n$ large enough, 
\begin{equation*}
\label{eq;ep-1} \A_h(\eta_n; S, S+\dl_n)  \le C_1 \dl_n \le \ep. 
\end{equation*}

Combining this with the lower semi-continuity of the action functional, for $n$ large enough,  
$$ \begin{aligned}
\A_h(\gm_n; 0, T) & \ge \A_h(\gm; 0, T) -\ep = \A_h(\xi; 0, S) +2\ep \\
			& = \A_h(\eta_n; 0, S+\dl_n) - \A_h(\eta_n; S, S+\dl_n) +2\ep \\
			& \ge \A_h(\eta_n; 0, S+\dl_n) +\ep. 
\end{aligned}
$$
This is a contradiction to the fact the $\gm_n$ is an $h$-free-time minimizer. 
\end{proof}


\section{Bi-hyperbolic solutions for anisotropic Kepler problem}

The main purpose of this section is to give a proof of Theorem \ref{thm;bi-hyper}, so we will only consider the planar anisotropic Kepler problem in this section. To introduce the proper topological constraint, we give the following definition. 

\begin{dfn}
Given any $T^- <T^+$, $\tht_{\pm} \in \rr$ and $r_i>0$, $i =1, 2$, we define 
$$ \Lmd^*_{T^{\pm}}(r_1 e^{i\tht_-}, r_2 e^{i\tht_+}) =\{ \xi \in H_{T^-, T^+}^1(r_1 e^{i\tht_-}, r_2 e^{i\tht_+}) : \; \xi(t) \ne 0,  \text{Arg}(\xi(T^-)) = \tht_-, \text{Arg}(\xi(T^+)) = \tht_+ \},$$
and $\Lmd_{T^{\pm}}(r_1 e^{i\tht_-}, r_2 e^{i\tht_+})$ the weak closure of $\Lmd^*_{T^{\pm}}(r_1 e^{i\tht_-}, r_2 e^{i\tht_+}) $ in $ H_{T^-, T^+}^1(r_1 e^{i\tht_-}, r_2 e^{i\tht_+})$.
\end{dfn}
Using the above definition, we further introduce the following set of paths
$$ \Lmd(r_1 e^{i \tht_-}, r_2 e^{i \tht_+}) = \cup_{-\infty < T^- < T^+ < \infty} \Lmd_{T^{\pm}}(r_1 e^{i\tht_-}, r_2 e^{i\tht_+}). $$
For simplicity, we set
$$ \Lmd_{T^{\pm}}(R, \tht_{\pm}) = \Lmd_{T^{\pm}}(R e^{i \tht_-}, R e^{i \tht_+}) \;  \text{ and } \; \Lmd(R, \tht_{\pm}) = \cup_{-\infty < T^- < T^+ < \infty} \Lmd_{T^{\pm}}(R, \tht_{\pm}). $$

\begin{lem}
\label{lem;free-minimizer-top} When $h>0$ and $r_1 e^{i \tht_-} \ne r_2 e^{i \tht_+}$,  there is a $\gm \in \Lmd_{T^{\pm}}(r_1 e^{i\tht_-}, r_2 e^{i \tht_+})$ and a constant $C_1$ independent of $\gm$, such that  
$$ \A_h(\gm; T^-, T^+) = \inf \{ \A_h(\xi): \xi \in \Lmd(r_1 e^{i \tht_-}, r_2 e^{i \tht_+}) \}.$$
In particular, if $\gm$ is collision-free, it is an $h$-energy solution of \eqref{eq;aniso-kepler}. 
\end{lem}

\begin{rem}
We call such a $\gm$ an $h$-free-time minimizer under topological constraint.  
\end{rem}

\begin{proof} The proof is similar to Lemma \ref{lem;exist-free-minimizer} and we omit the details here. 




\end{proof}


Given arbitrarily a positive energy constant $h$, a pair of angles $\tht_{\pm} \in \rr$ and a sequence of positive numbers $\{ R_n \nearrow \infty\}_{n =1}^{\infty}$. By Lemma \ref{lem;free-minimizer-top}, there is a sequence of paths $\{\gm_n  \in \Lmd_{T_n^{\pm}}(R_n, \tht_{\pm})\}_{n=1}^{\infty}$ with each $\gm_n$ being an $h$-free-time minimizer under topological constraint. Without loss of generality, we further assume  
\begin{equation}
    \label{eq;rho-n} 0 \in (T_n^-, T_n^+) \text{ with } \rho_n = |\gm_n(0)| = \inf \{|\gm_n(t)|:\; t \in [T_n^-, T_n^+]\}.
\end{equation}

Following the approach given in \cite{FT00}, \cite{BDT17} and \cite{BBD18}, we obtain the following result.  
\begin{lem} \label{lem;rho-bound}
When $|\tht_+ -\tht_-| >\pi$ and each $\gm_n$ is collision-free, $\liminf_{n \to \infty} \rho_n <\infty.$
\end{lem}

\begin{proof}
By a contradiction argument, after passing to a  sub-sequence, we may assume 
$$ \lim_{n \to \infty} \rho_n =\infty \text{ and } d = \lim_{n \to \infty} \rho_n /R_n \in [0, 1]. $$

First let's consider the case that $d >0$. By Lemma \ref{lem;I-ge}, $R_n^2 \ge 2h (T^{\pm}_n)^2 + \rho_n^2$. This implies
$$ |T_n^{\pm}| \le \frac{1}{\sqrt{2h}} \left( 1 -\left( \frac{\rho_n}{R_n} \right)^2 \right)^{\ey} R_n. $$
By the energy identity,
\begin{equation*} 
\begin{aligned}
\A_h(\gm_n; T_n^-, T_n^+) & = 2 \int_{T_n^-}^{T_n^+} h + U(\gm_n(t)) \,dt \le 2 \left( h + \frac{U_{max}}{\rho_n^{\al}} \right) (T^+_n -T^-_n) \\
& \le \frac{4}{\sqrt{2h}} \left( h + \frac{U_{max}}{\rho_n^{\al}} \right)  \left( 1 -\left( \frac{\rho_n}{R_n} \right)^2 \right)^{\ey} R_n
\end{aligned}
\end{equation*} 
Our assumptions then imply the existence of a $\dl \in (0, 1)$, such that when $n$ is large enough,
\begin{equation} \label{eq;Ah-xn>Rn}
\A_h(\gm_n; T_n^-, T_n^+)  \le 2 \sqrt{2h} (1 - \dl) R_n.
\end{equation}
Meanwhile as $ |\tht_+ - \tht_-|>\pi$, we can find two instants $T_n^- < t_n^- \le t_n^+ < T_n^+$ satisfying
$$ \langle \gm_n(t_n^{\pm}), \gm_n(T_n^{\pm}) \rangle =0,$$
Then the triangle inequality implies 
$$ |\gm_n(T_n^{\pm}) - \gm_n(t_n^{\pm})| \ge |\gm_n(T_n^{\pm}) -0| \ge R_n. $$
Using this, we can get 
$$  
\A_h(\gm_n; T_n^-, T_n^+)  \ge \int_{T_n^-}^{t_n^-} + \int_{t_n^+}^{T_n^+} \ey |\dot{\gm}_n|^2 + h \,dt  \ge \sqrt{2h}   \left( \int_{T_n^-}^{t_n^-}  + \int_{t_n^+}^{T_n^+} |\dot{\gm}_n|  \,dt \right) \ge 2 \sqrt{2h} R_n.  
$$
This is a contradiction to \eqref{eq;Ah-xn>Rn}. 

Now let's consider the case $d=0$. For each $n$, set 
$$ y_n(\tau) = \rho_n^{-1} \gm_n(\rho_n \tau), \; \forall \tau \in [T_n^-/\rho_n, T_n^+/\rho_n]. $$
By a direct computation,
\begin{equation}
\label{eq;yn''} \begin{cases}
& y''_n(\tau) = \frac{\nabla U(y_n(\tau))}{\rho_n^{\al}  }; \\
& \ey |y'_n(\tau)|^2 - \frac{U(y_n(\tau))}{\rho_n^{\al} } = h.
\end{cases}
\end{equation} 

After passing to a proper sub-sequence, $y_n \to y_{\infty}$ in $C^2_{\text{loc}}(\rr, \rr^2)$, where 
$$ \begin{cases}
y''_{\infty}(\tau) & = 0; \\
\ey |y_{\infty}'(\tau)|^2 & = h. 
\end{cases}
$$
Since $|y_n(0)| = 1$, we may further assume $y_n(0) \to e^{i\tht_0}$, as $n \to \infty$, for some $\tht_0 \in \rr$. Then  
\begin{equation*}
\label{eq;y-infty}  y_{\infty}(\tau) =  \begin{cases}
e^{i \tht_0} + \sqrt{2h} e^{i(\tht_0 +\frac{\pi}{2})} \tau, & \text{ if } \tht_- < \tht_+; \\
 e^{i \tht_0} + \sqrt{2h} e^{i(\tht_0 - \frac{\pi}{2})} \tau, & \text{ if } \tht_- > \tht_+.  
\end{cases}   
\end{equation*}
As a result, 
\begin{equation}
\label{eq;s-infty} \lim_{\tau \to + \infty} |\text{Arg}(y_{\infty}(\tau)) - \text{Arg}(y_{\infty}(-\tau))| = \pi. 
\end{equation}

To get a contradiction, set 
$$ \tht_{\infty}(\tau)= \text{Arg}(y_{\infty}(\tau)), \; \tht_n(\tau)= \text{Arg}(y_n(\tau)). $$ 
Then for any $\ep>0$ and $\tau_1$, there is an integer $n_0= n_0(\ep, \tau_1)$, such that   
\begin{equation*}
|\tht_n(\tau_1)-\tht_{\infty}(\tau_1)| \le \ep/2, \; \forall n \ge n_0. 
\end{equation*}
Meanwhile $\lim_{n \to \infty} \rho_n = \infty$ implies the existence of a $\tau_{\ep}>0$ large enough, such that 
\begin{equation*}
\label{eq;tau-ep} \frac{2}{\al} \frac{\sqrt{2h + 2 U_{max} \rho_n^{-\al}}}{(2h)^{\frac{2+\al}{4}}} \tau_{\ep}^{-\frac{\al}{2}} \le \ep/2, \;\; \forall n. 
\end{equation*}

Let  $s_n(t) = \gm_n(t)/|\gm_n(t)|$. By a direct computation, 
$$ \int_{\tau_{\ep}}^{T_n^+/\rho_n} |\tht'_n(\tau)| \,d\tau = \int_{\rho_n \tau_{\ep}}^{T^+} |\dot{s}_n(t)| \,dt 
. $$
Then Lemma \ref{lem;|s-dot|} implies that for any $\tau_1 \ge \tau_{\ep}$,
$$ |\tht_n(\tau_1)-\tht_+| \le \int_{\tau_{\ep}}^{T_n^+/\rho_n} |\tht'_n(\tau)| \,d\tau \le \frac{2}{\al} \frac{\sqrt{2h + 2 U_{max} \rho_n^{-\al}}}{(2h)^{\frac{2+\al}{4}}} \tau_{\ep}^{-\frac{\al}{2}} \le \ep/2.
$$

The above estimates show
$$ |\tht_{\infty}(\tau_1) - \tht_+| \le |\tht_{\infty}(\tau_1) -\tht_n(\tau_1)| + |\tht_n(\tau_1)- \tht_+| \le \ep. $$
As this holds for any $\ep>0$ small and $\tau_1 \ge \tau_{\ep}$, we get  
$$ \lim_{\tau_1 \to +\infty} \tht_{\infty}(\tau_1)=\tht_+. $$
A similar argument as above shows 
$$ \lim_{\tau_1 \to -\infty} \tht_{\infty}(\tau_1)=\tht_-. $$
As a result, $\lim_{\tau \to +\infty}|\tht_{\infty}(\tau) - \tht_{\infty}(-\tau)| = |\tht_+ -\tht_-| >\pi$, which is a contradiction to \eqref{eq;s-infty}.

\end{proof}

The next result tells us when a minimizer under the topological constraint is collision-free. 
\begin{prop}
\label{prop;BTV}{\cite[Theorem 2]{BTV13}} Under condition \eqref{eq;U-nondeg}, for any $\psi_1 < \psi_2$ with $\psi_i  \in \mathcal{M}(\tilde{U})$, $i =1, 2$,  there exists a constant $\bar{\al} = \bar{\al}(U, \psi_1, \psi_2) \in (0, 2)$, such that when $\al >\bar{\al}$, for any $[\tht_1, \tht_2] \subset [\psi_1, \psi_2]$, if $\gm \in \Lmd_{T^{\pm}}(r_1e^{i\tht_1}, r_2 e^{i\tht_2})$ satisfies 
$$ \A(\gm; T^-, T^+) \le \A(\xi; T^-, T^+), \; \forall \xi \in \Lmd_{T^{\pm}}(r_1e^{i\tht_1}, r_2 e^{i\tht_2}),$$
then $\gm$ is collision-free. 
\end{prop}

\begin{proof}[Proof of Theorem \ref{thm;bi-hyper}]
Without loss of generality, let's assume $\tht_- < \tht_+$. Then we can find a largest $\psi_1 \le \tht_-$ and a smallest $\psi_2 \ge \tht_+$  satisfying $\psi_i \in \mathcal{M}(\tilde{U})$, $i =1, 2$. Let $\bar{\al}$ be the constant given by Proposition \ref{prop;BTV}.

Choose a sequence of positive numbers $\{ R_n \nearrow \infty \}_{n =1}^{\infty}$ with each $R_n>1$. By Lemma \ref{lem;free-minimizer-top}, there is a sequence of paths $\{ \gm_n = r_n e^{i \tht_n} \in \Lmd_{T_n^{\pm}}(R_n, \tht_{\pm}) \}_{n =1}^{\infty}$ with each $\gm_n$ being an $h$-free-time minimizer under topological constraint. By Proposition \ref{prop;BTV}, when $\al \in (\bar{\al}, 2)$, each $\gm_n$ is a collision-free $h$-energy solution of \eqref{eq;aniso-kepler}. Meanwhile $\psi_i$, $i =1, 2$, being global minimizers of $\tilde{U}$ imply
\begin{equation}
\label{eq;psi1-pasi2} \tht_n(t) \in [\psi_1, \psi_2], \; \forall t \in [T_n^-, T_n^+], \; \forall n. 
\end{equation}

Let's further assume \eqref{eq;rho-n} holds, for each $n$. As $\tht_+ -\tht_- >\pi$, by Lemma \ref{lem;rho-bound}, after passing to a proper sub-sequence, we have  
\begin{equation}
	\label{eq;rho-n-bound} \rho = \lim_{n \to \infty} \rho_n < \infty \text{ and } \rho_n = |\gm_n(0)| \le 2\rho, \; \forall n.
\end{equation}

Given any $[T^{-}, T^+] \subset [T^-_n, T_n^+]$ satisfying $|\gm_n(T^{\pm})| >1$. By property (a) and (c) in Lemma \ref{lem;phh-le}, there exist two paths
$$ \xi^-_n \in H^1_{S_n^-, 0}(\gm_n(T^-), 0) \text{ with } \frac{\xi^-_n(t)}{|\xi^-_n(t)|} = \frac{\gm_n(T^-)}{|\gm_n(T^-)|},\, \; \forall t; $$
$$ \xi^+_n \in H^1_{0, S_n^+}(0, \gm_n(T^+)) \text{ with } \frac{\xi^+_n(t)}{|\xi^-_n(t)|} = \frac{\gm_n(T^+)}{|\gm_n(T^+)|},\, \; \forall t,$$
and two positive constants $C_1$ and $C_2$ independent of $n$, such that 
$$ \A_h(\xi_n^-; S_n^-, 0) + \A_h(\xi_n^+; 0, S^+_n) \le C_1 + C_2\max\{|\gm_n(T^-)|, |\gm_n(T^+)| \}. $$
Since  $\gm_n|_{[T^-, T^+]}$ is an $h$-free-time minimizer under topological constraint, this implies 
\begin{equation}
	\label{eq;phih-le-2}  
		\A_h(\gm_n; T^-, T^+) \le C_1 + C_2\max\{|\gm_n(T^-)|, |\gm_n(T^+)| \}
\end{equation}
Using this we can show 
$$ \lim_{n \to \infty}T_n^{\pm} = \pm \infty. $$
We only give the details for $T^{+}_n$. By a similar argument as in the proof of Theorem \ref{thm;hyper},
\begin{equation} \label{eq;Tn+}
	\begin{aligned}
		\left( |\gm_n(T_n^+)| - |\gm_n(0)| \right)^2 & \le |\gm_n(T_n^+) - \gm_n(0)|^2 \le \left( \int_0^{T^+_n} |\dot{\gm}_n| \,dt \right)^2  \le T_n^+ \int_0^{T_n^+} |\dot{\gm}_n|^2 \,dt \\
		&  = T_n^+ \A_h(\gm_n; 0, T^+_n)  \le T_n^+\A_h(\gm_n; T_n^-, T_n^+) \\
		& \le T_n^+(C_1 + C_2 R_n).
	\end{aligned}
\end{equation}
This implies 
$$ T_n^+ \ge \frac{(R_n - \rho_n)^2}{C_1 + C_2 R_n} \to \infty, \text{ as } n \to \infty. $$

We claim for any $T>0$, there exist two positive constants $C_3, C_4$ depending on $C_1, C_2$ and $\rho$, but independent of $T$ and $n$, such that
\begin{equation} \label{eq;gmn-T}
	 |\gm_n(\pm T)|  \le C_3 + C_4 T. 
\end{equation}

The claim is trivial if $|\gm_n(T)| \le 1$. Let's assume $|\gm_n(T)| >1$. Notice that for $n$ large enough, there is a unique $S <0$ with $|\gm_n(S)| = |\gm_n(T)|>1$. Then similar estimates as in \eqref{eq;Tn+} show 
\begin{equation}
 \begin{aligned}
 (|\gm_n(T)|- \rho_n)^2 & \le |\gm_n(T)- \gm_n(0)|^2 \le \left( \int_0^T |\dot{\gm}_n| \right)^2 \le T \int_0^T |\dot{\gm}_n|^2 \,dt \\
 & = T \A_h(\gm_n; 0, T) \le T \A_h(\gm_n; S, T) \le T(C_1 + C_2 |\gm_n(T)|),
 \end{aligned}
 \end{equation} 
where the last inequality follows from \eqref{eq;phih-le-2}. Solving the above inequality for $|\gm_n(T)|$ show
 $$  |\gm_n(T)| \le \frac{2 \rho_n + C_2 T + \sqrt{4(C_ 1 + C_2 \rho_n) T + C_2^2 T^2}}{2} \le \frac{4\rho + C_2 T + \sqrt{4(C_1 +2C_2 \rho)T + C_2 T^2}}{2}. 
 $$
A similar argument shows the above inequality holds for $|\gm_n(-T)|$ as well. This proves our claim. 

Then \eqref{eq;phih-le-2} and  \eqref{eq;gmn-T} imply 
 $$ \int_{-T}^T |\dot{\gm}_n|^2 \,dt = \A_h(\gm_n; -T, T) \le C_1 +C_2(C_3 +C_4 T). $$
This means for any $T>0$, $\{\gm_n|_{[-T, T]}\}$ is a bounded sequence in $H^1_{-T, T}$, for any $T>0$, when $n$ is large enough. By a diagonal argument, similar as in the proof of Theorem \ref{thm;hyper}, we can find a 
$$ \gm(t) = r(t)e^{i\tht(t)} \in H^1(\rr, \cc),$$
such that $\gm_n$ converges to $\gm$ weakly in $H^1_{loc}$-norm and strongly in $L^{\infty}_{loc}$-norm. Then \eqref{eq;psi1-pasi2} implies 
\begin{equation}
\label{eq;psi1-psi2-tht} \tht(t) \in [\psi_1, \psi_2], \; \forall t \in \rr. 
\end{equation}

\begin{lem}
\label{lem;gm-free-minimizer-2}  $\forall T>0$, $\gm|_{[-T, T]}$ is an $h$-free-time minimizer under topological constraint, i.e.,
$$ \A_h(\gm; -T, T) = \inf \{ \A_h(\xi): \; \forall \xi \in \Lmd( r(-T) e^{i \tht(-T)}, r(T)e^{i \tht(T)}) \}. $$
\end{lem}
Postpone the proof of the above lemma for a moment. By Proposition \ref{prop;BTV},  $\gm|_{[-T, T]}$ must be collision-free, and hence an $h$-energy solution of \eqref{eq;aniso-kepler}. Then $\gm_n(t)$ converges to $\gm(t)$ in $C^2_{loc}$. For any $\ep>0$ and $t_1>0$, we can find a $n_0=n_0(\ep, t_1)$ large enough, such that 
$$ |\tht_n(t_1) -\tht(t_1)| \le \ep/2, \; \forall n \ge n_0. $$
By \eqref{eq;rho-n-bound}, we can find a $t_{\ep}>0$ large enough, such that 
$$ \frac{2}{\al} \frac{\sqrt{2h \rho_n^{\al} + 2U_{max}}}{(2h)^{\frac{2+\al}{4}}} t_{\ep}^{-\frac{\al}{2}} \le \frac{\ep}{2}, \;\; \forall n. $$
Then Lemma \ref{lem;|s-dot|} implies that for any $t_1 \ge t_{\ep}$, 
$$ |\tht_n(t_1)-\tht_n(T^+_n)| \le \int_{t_{\ep}}^{T_n^+} |\dot{\tht}_n(t)|\,dt \le  \frac{2}{\al} \frac{\sqrt{2h \rho_n^{\al} + 2U_{max}}}{(2h)^{\frac{2+\al}{4}}} t_{\ep}^{-\frac{\al}{2}} \le \frac{\ep}{2}.
$$
These estimates then imply
$$ |\tht(t_1)- \tht_+| \le |\tht(t_1) - \tht_n(t_1)| + |\tht_n(t_1)- \tht_n(T_n^+)| \le \ep. $$
As the above inequality holds for any $\ep>0$ and $t_1 \ge t_{\ep}$, we get 
$$ \lim_{t \to +\infty} \tht(t) = \tht_+. $$
A similar argument shows 
$$ \lim_{t \to -\infty} \tht(t) = \tht_-. $$
The desired result now follows from Property (b) in Theorem \ref{thm;asym}.
\end{proof}

\begin{proof}[Proof of Lemma \ref{lem;gm-free-minimizer-2}]
The proof is similar to Lemma \ref{lem;gm-free-minimizer}. By a contradiction argument, let's assume there is an $\ep>0$ and a path 
$$ \xi \in \Lmd_{S^{\pm}}( r(-T) e^{i \tht(-T)}, r(T)e^{i \tht(T)}), $$
such that 
\begin{equation*}
\label{eq;4-ep}  \A_h(\gm; -T, T) - A_h(\xi; S^-, S^+) = 4 \ep  >0. 
\end{equation*}
For each $n$, set $\dl_n= |\gm_n(T) - \gm(T)|$ and $\dl_n^* = |\gm_n(-T) - \gm_n(-T)|$ and define a path
\begin{equation*}
\eta_n(t) = \begin{cases}
\xi(S^-) + \frac{t-S^- +\dl_n^*}{\dl_n^*} (\xi(S^-) - \gm_n(-T)), & \text{ when } t \in [S^- -\dl_n^*, S^-]; \\
\xi(t), & \text{ when } t \in [S^-, S^+]; \\
\xi(S^+) + \frac{t - S^+}{\dl_n} (\gm_n(T) -\xi(S^+)), & \text{ when } t \in [S^+, S^+ +\dl_n]. 
\end{cases}
\end{equation*}
Such a definition implies $\eta_n \in \Lmd(\gm_n(-T), \gm_n(T))$. 

As both $\dl_n$ and $\dl_n^*$ goes to zero, when $n \to \infty$, we can find a constant $C_1>0$ independent of $n$, such that for $n$ large enough
\begin{equation*}
\A_h(\eta_n; S^--\dl_n^*, S^-) \le C_1 \dl_n^* \le \ep \text{ and } \A_h(\eta_n; S^+, S^+ + \dl_n) \le C_1 \dl_n \le \ep. 
\end{equation*}
Combining these with the lower semi-continuity of the action functional, for $n$ large enough,
$$ \begin{aligned}
\A_h(\gm_n; & -T, T) \ge \A_h(\gm; -T, T) -\ep = \A_h(\xi; S^-, S^+) +3\ep \\
& = \A_h(\eta_n; S^- -\dl_n^*, S^+ +\dl_n) - \A_h(\eta_n;  S^- -\dl_n^*, S^-) - \A_h(\eta_n; S^+, S^+ +\dl_n) +3\ep \\
& \ge \A_h(\eta_n; S^- -\dl_n^*, S^+ +\dl_n) +\ep.
\end{aligned}
$$
This is a contradiction to the fact that $\gm_n$ is an $h$-free-time minimizer under topological constraint. 

\end{proof}


\section{rule out collision for fixed-time local minimizers} \label{sec;coll-free}

The last two sections will focus on Gutzwiller's anisotropic Kepler problem. In this section we give a proof of Proposition \ref{prop;coll-free-2}. The approach follows ideas from celestial mechanics. First we show collisions must be isolated following the approach given by Chenciner in \cite{C02}, then we will blow-up the solution near an isolated collision singularity following Ferrario and Terracini \cite{FT04}, at last we get a contradiction by a local deformation following Montgomery \cite{Mont99a}. In the first two steps, we only assume condition \eqref{eq;U-homo}.

Let $x \in H^1_{t_1, t_2}(p, q)$ be a fixed-time local minimizer of $\A$ in the rest of the section.

\begin{lem} \label{lem;coll-iso} 
Under condition \eqref{eq;U-homo}, $coll(x) = \{ t \in (t_1, t_2): \; x(t) =0\}$ is isolated in $[t_0, t_1]$. 
\end{lem}

\begin{proof}
Since $\A(x; t_1, t_2)$ is finite, the Lebesgue measure of $coll(x)$ must be zero. Then for any $(a, b) \subset (t_1, t_2)\setminus coll(x)$, $x|_{(a, b)}$ is smooth solution of \eqref{eq;aniso-kepler}. This means the energy function below is well-defined for all $t \notin coll(x)$ and therefore almost everywhere in $(t_1, t_2)$
$$ h(t) = \ey |\dot{x}(t)|^2 -U(x(t)). $$ 
Since $x$ is a minimizer, we can further show $h(t)$ is constant for all $t \notin coll(x)$ (see \cite[page 7]{C02}). 

Assuming $t_0 \in coll(x)$ is not isolated, there must be a sequence of non-intersecting intervals $(a_i, b_i)$ with both $a_i$ and $b_i$ converging to $t_0$, as $i$ goes to infinity, and satisfying
\begin{equation}
\label{eq;I-a-b} x(a_i) = x(b_i) =0, \; x(t) \ne 0, \forall t \in (a_i, b_i).
\end{equation}
By Lemma \ref{lem;Lag-Jacobi}, 
$$ \ddot{I}(x(t)) = 4h(t) + \frac{2 (2-\al)U(s(t))}{|x(t)|^{\al}} \to \infty, \text{ as } t \to t_0.$$ 
Meanwhile by \eqref{eq;I-a-b}, we can find a sequence of instants $t_i \in (a_i, b_i)$ satisfying $\dot{I}(t_i) =0$, which is absurd. 
\end{proof}

Without loss of generality, let's assume $0 \in (t_1, t_2)$ and $x(0)=0$. By Lemma \ref{lem;coll-iso}, there is a $\dl>0$ small enough such that both $x|_{(-\dl, 0)}$ and $x|_{(0, \dl)}$ are collision-free solutions of \eqref{eq;aniso-kepler}.

\begin{dfn}
\label{dfn;blow-up} For any $\lmd>0$, we define the $\lmd$ blow-up of $x|_{(-\dl, \dl)}$ as 
$$ x^{\lmd}(t) = \lmd^{-\frac{2}{2+\al}} x(\lmd t), \; t \in (-\dl/\lmd, \dl/\lmd). $$
\end{dfn}

\begin{lem}
\label{lem;action-homo} $\A(x^{\lmd}; -\dl/\lmd, \dl/\lmd) = \lmd^{-\frac{2-\al}{2+\al}} \A(x; -\dl, \dl)$. 
\end{lem}

\begin{proof}
	As the potential function is $-\al$ homogeneous, this follows from a direct computation.
\end{proof}

\begin{lem}
\label{lem;blow-up} If $s(\pm \lmd_n) \to s^{\pm}$, as $\lmd_n \to 0$. For any $T>0$, the sequences $x^{\lmd_n}$ and $\dot{x}^{\lmd_n}$ converge to $\bar{x}$ and its derivative $\dot{\bar{x}}$ uniformly on $[-T, T]$ and on compact subsets of $[-T, T] \setminus \{0\}$ respectively, where $\bar{x}(t)$ is defined as
\begin{equation}
\label{eq;xbar} \bar{x}(t) = \begin{cases}
(\kp_+ t)^{\frac{2}{2 +\al}} s^+, & \text{ if } t \ge 0; \\
(-\kp_- t)^{\frac{2}{2 +\al}} s^-, & \text{ if } t \le 0.
\end{cases}
\end{equation}
\end{lem}

\begin{rem}
By a simple computation, $\bar{x}(t)$ is a zero energy solution of \eqref{eq;aniso-kepler}. 
\end{rem}

\begin{proof}
	With Proposition \ref{prop;Sundman-est}, the result can be proven as \textbf{(7.4)} in \cite{FT04}. 
\end{proof}

\begin{lem}
\label{lem;blow-up-homo} For any $T>0$, there exist{}s a sequence of positive integers $\{N_n \nearrow \infty\}$, as $n \to \infty$, such that for any $\phi \in H^1_{-T, T}(\rr^d)$ which are $C^1$ in small neighborhoods of $T$ and $-T$, 
$$ \lim_{n \to \infty} \A(x^{\lmd_n} + \phi + \psi_n; -T, T) - \A(x^{\lmd_n};-T, T) = \A(\bar{x} + \phi; -T, T) - \A(\bar{x}; -T, T),$$
where $\psi_n \in H^1_{-T, T}(\rr^d)$ is 
\begin{equation}
\label{eq;psi-n} \psi_n(t) = \begin{cases}
N_n(T-t) (\bar{x}(t) - x^{\lmd_n}(t)), \; & \text{ if } t \in [T-N_n^{-1}, T]; \\
\bar{x}(t) - x^{\lmd_n}(t), \; & \text{ if } t \in [-T + N_n^{-1}, T - N_n^{-1}]; \\
N_n(T +t) (\bar{x}(t) - x^{\lmd_n}(t)), \; & \text{ if } t \in [-T, -T + N_n^{-1}].
\end{cases}
\end{equation}
\end{lem}

\begin{proof}
The proof is the same as Proposition \textbf{(7.9)} in \cite{FT04}.
\end{proof}

We will only consider Gutzwiller's anisotropic Kepler problem for the rest of the section. Let's introduce the $M$ weighted inner product and norm as 
$$ \me{x, y}:= <x, M y>, \; \|x\|= \me{x, x}^{\ey}, \; \forall x, y \in \rr^d,
$$
where $M = \text{diag}(m_1, \dots, m_d)$ is the diagonal matrix. For simplicity, we further assume 
$$ m_1 \le m_2 \le \cdots \le m_d.$$
Let $\I_j$, $1 \le j \le k$, be a partition of  $\{1, 2, \cdots, d\}$ satisfying $\{i_1, i_2 \} \subset \I_j$ iff $m_{i_1} = m_{i_2}$, and 
$$  m_{i_1} < m_{i_2}, \text{ if } i_1 \in \I_{j_1}, i_2 \in \I_{j_2} \text{ with } j_1 < j_2. $$
Let $\{\ee_i: 1 \le i \le d\}$ being an orthonormal basis of $\rr^d$, we denote $\rr^{\I_j}$ as the linear subspace of $\rr^d$ with basis $\{\ee_i: i \in \I_j\}$. Then the set of critical points of $U$ on $\bs$ is
\begin{equation}
\label{eq;Cr-U-aniso} Cr(U) = \cup_{j =1}^k \Big( \bs \cap \rr^{\I_j} \Big).
\end{equation}

For any $T>0$ and $\sg \in \bs$,  define a local deformation of $\xb$ near the singularity as
\begin{equation}
\label{eq;xb-ep} \xb^{\ep}(t) = \xb(t) + f_{\ep}(t) \sg, \; t \in [-T, T],
\end{equation}
where $f_{\ep}: [-T, T] \to \rr$ is an even function with 
\begin{equation}
\label{eq;f-ep} f_{\ep}(t) = \begin{cases}
0, \; & \text{ when } t \in [\ep^{\frac{2 +\al}{2}}+ \ep, T]; \\
\ep + \ep^{\frac{2 +\al}{2}} -t, \; & \text{ when } t \in [\ep^{\frac{2 +\al}{2}}, \ep^{\frac{2 +\al}{2}}+ \ep]; \\
\ep, \; & \text{ when } t \in [0, \ep^{\frac{2 +\al}{2}}].
\end{cases}
\end{equation} 
The above local deformation and the following computations are essentially the same as those used by Montgomery in an unpublished paper \cite{Mont99a}. By the definition of $\xb$ and $\xb^{\ep}$,
$$
 \begin{aligned}
& U(\xb^{\ep}(t))  - U(\xb(t))  =  \|\xb(t) + f_{\ep}(t) \sg \|^{-\al} - \|\xb(t)\|^{-\al} \\ 
& = \Big( (\kp_+ t)^{\frac{4}{2 +\al}} \|s_+\|^2 + f_{\ep}^2(t) \|\sg\|^2 + 2 (\kp_+ t)^{\pwa} f_{\ep}(t) \me{s^+, \sg}  \Big)^{-\frac{\al}{2}} - (\kp_+ t)^{-\frac{2\al}{2 +\al}} \|s^+\|^{-\al}.
\end{aligned}
$$
$$ \begin{aligned}
\A(\xb^{\ep}; 0, T) - \A(\xb; 0, T) & = \ey \int_{\ep^{\pwb}}^{\ep^{\pwb}+ \ep} |\dot{\xb}^{\ep}|^2 - |\dot{\xb}|^2 \,dt + \int_{0}^{\ep^{\pwb}}+ \int_{\ep^{\pwb}}^{\ep^{\pwb}+\ep}  U(\xb^{\ep}) - U(\xb) \,dt \\
& := A_1 + A_2 + A_3.
\end{aligned}
$$
The following expressions of $A_i$ and $B_i$ will be needed in our proofs. 
\begin{equation*}
\begin{aligned}
\A(\xb^{\ep}; -T, 0) - \A(\xb; -T, 0) & =  \ey \int^{-\ep^{\pwb}}_{-\ep^{\pwb}-\ep} |\dot{\xb}^{\ep}|^2 - |\dot{\xb}|^2 \,dt + \int^{0}_{-\ep^{\pwb}}+ \int^{-\ep^{\pwb}}_{-\ep^{\pwb}-\ep}  U(\xb^{\ep}) - U(\xb) \,dt \\
& := B_1 + B_2 + B_3.
\end{aligned}
\end{equation*}
\begin{equation}
\label{eq;A1}A_1  = \ey \int_{\ep^{\pwb}}^{\ep^{\pwb}+\ep} |\dot{\xb} - \sg|^2 - |\dot{\xb}|^2 \,dt =  \int_{\ep^{\pwb}}^{\ep^{\pwb}+\ep}  \ey - \pwa \kp_+^{\pwa} t^{\frac{-\al}{2 +\al}} \langle \sg, s^+ \rangle \,dt. 
\end{equation}

\begin{equation}
\label{eq;B1}B_1  = \ey \int^{-\ep^{\pwb}}_{-\ep^{\pwb}-\ep} |\dot{\xb} - \sg|^2 - |\dot{\xb}|^2 \,dt = \int^{-\ep^{\pwb}}_{-\ep^{\pwb}-\ep} \ey - \pwa \kp_-^{\pwa} |t|^{\frac{-\al}{2 +\al}} \langle \sg, s^- \rangle \,dt. 
\end{equation}


\begin{equation}
\label{eq;A2} \begin{aligned}
A_2 &= \int_{0}^{\ep^{\pwb}} U(\xb^{\ep})  - U(\xb) \,dt \;  (\text{set } \tau = \ep^{-1} t^{\pwa}) \\
& = \frac{2+\al}{2} \ep^{\frac{2 -\al}{2}} \int_0^1 \frac{\tau^{\al/2}}{\left( \kp_+^{\frac{4}{2+\al}} \|s^+\|^2 \tau^2 + 2 \kp_+^{\pwa} \me{s^+, \sg} \tau +\|\sg\|^2 \right)^{\frac{\al}{2}}} - \frac{\tau^{\al/2}}{\kp_+^{\frac{2\al}{2 +\al}} \|s^+\|^{\al} \tau^{\al}} \,d\tau
\end{aligned}
\end{equation}
\begin{equation}
\label{eq;B2} \begin{aligned}
B_2 &= \int^{0}_{-\ep^{\pwb}} U(\xb^{\ep})  - U(\xb) \,dt \;  (\text{set } \tau = \ep^{-1} (-t)^{\pwa}) \\
& = \frac{2+\al}{2} \ep^{\frac{2 -\al}{2}} \int_0^1 \frac{\tau^{\al/2}}{\left( \kp_-^{\frac{4}{2+\al}} \|s^-\|^2 \tau^2 + 2 \kp_-^{\pwa} \me{s^-, \sg} \tau +\|\sg\|^2 \right)^{\frac{\al}{2}}} - \frac{\tau^{\al/2}}{\kp_-^{\frac{2\al}{2 +\al}} \|s^-\|^{\al} \tau^{\al}} \,d\tau
\end{aligned}
\end{equation}


\begin{equation}
\label{eq;A3}   A_3  =  \int_{\ep^{\pwb}}^{\ep^{\pwb}+\ep} \frac{1}{\Big( (\kp_+ t)^{\frac{4}{2 +\al}} \|s^+\|^2 + f_{\ep}^2(t) \|\sg\|^2 + 2 (\kp_+ t)^{\pwa} f_{\ep}(t) \me{s^+, \sg}  \Big)^{\frac{\al}{2}} } - \frac{1}{(\kp_+ t)^{\frac{2\al}{2 +\al}} \|s^+\|^{\al}} \,dt
\end{equation}

\begin{equation}
\label{eq;B3}   B_3  =  \int^{-\ep^{\pwb}}_{-\ep^{\pwb}-\ep} \frac{1}{\Big( (\kp_- |t|)^{\frac{4}{2 +\al}} \|s^-\|^2 + f_{\ep}^2(t) \|\sg\|^2 + 2 (\kp_- |t|)^{\pwa} f_{\ep}(t) \me{s^-, \sg}  \Big)^{\frac{\al}{2}} } - \frac{1}{(\kp_- |t|)^{\frac{2\al}{2 +\al}} \|s^-\|^{\al}} \,dt
\end{equation}

\begin{lem}
\label{lem;homothetic-not-minimizer} Under condition \eqref{eq;U-aniso}, for any zero energy homothetic solution $\xb(t)$ as \eqref{eq;xbar}, there is an $\sg \in \bs$ and $\xb^{\ep} \in H^1_{-T, T}(\xb(-T), \xb(T))$ as \eqref{eq;xb-ep}, such that for $\ep>0$ small enough,
$$ \A(\xb^{\ep}; -T, T) < \A(\xb; -T, T).$$
\end{lem}

\begin{proof}
By Proposition \ref{prop;Sundman-est}, $s^{\pm} \in Cr(U)=\cup_{j =1}^k \Big( \bs \cap \rr^{\I_j} \Big)$. Consider three different cases as below. 

\emph{Case 1}: $s^- \in \bs \cap \rr^{\I_{j_1}}$ and $s^+ \in \bs \cap \rr^{\I_{j_2}} \text{ with } j_1 \ne j_2.$

As a result, $\langle s^-, s^+ \rangle = \me{s^-, s^+}=0$. Let $\sg = s^-$, then \eqref{eq;A1} implies
$$ A_1 = \int_{\ep^{\pwb}}^{\ep^{\pwb}+\ep}  \ey \,dt = \ey \ep. $$
Meanwhile by \eqref{eq;A3} $A_3 < 0$, and by \eqref{eq;A2}, there is a constant $C_1 >0$, such that 
$$ 
A_2 = \frac{2+\al}{2} \ep^{\frac{2 -\al}{2}} \int_0^1 \frac{\tau^{\al/2}}{\left( \kp_+^{\frac{4}{2+\al}} \|s^+\|^2 \tau^2 + \|s^-\|^2 \right)^{\frac{\al}{2}}} - \frac{\tau^{\al/2}}{\kp_+^{\frac{2\al}{2 +\al}} \|s^+\|^{\al} \tau^{\al}} \,d\tau \le -C_1 \ep^{\frac{2-\al}{2}}
$$
Since $\frac{2-\al}{2} \in (0, 1)$, when $\ep>0$ is small enough, 
$$ \A(\xb^{\ep};0, T) -\A(\xb; 0, T) \le \ey \ep - C_1 \ep^{\frac{2-\al}{2}} < 0$$

As $\langle s^-, s^-\rangle =1$ and $\me{s^-, s^-} >0$, similarly by \eqref{eq;B1}, \eqref{eq;B2} and \eqref{eq;B3}, we have 
$$ \A(\xb^{\ep};-T, 0) -\A(\xb; -T, 0) \le \ey \ep - C_1 \ep^{\frac{2-\al}{2}} < 0$$

\emph{Case 2}: $s^{\pm} \in \bs \cap \rr^{\I_{j_0}}$ with $|\I_{j_0}|=1$. 

Then  $\bs \cap \rr^{\I_{j_0}} = \{ \pm \ee_{i_0} \}$, for some $i_0$. We may choose $\sg = \ee_{i_1}$ for any $i_1 \ne i_0$. Then 
$$ \langle \ee_{i_0}, \ee_{i_1} \rangle = \me{\ee_{i_0}, \ee_{i_1}} =0$$
and the rest follows from a similarly argument as above.

\emph{Case 3}: $s^{\pm} \in \bs \cap \rr^{\I_{j_0}}$ with  $|\I_{j_0}| \ge 2$. 

In this case, $\xb(\rr)$ is contained entirely in a two-dim linear subspace of $\rr^{\I_{j_0}}$. As a result, we may choose $\sg$ to be a normalized vector from this plan satisfying
$$ \langle \sg, s^{\pm} \rangle \ge 0.$$
This then implies 
$$ \me{\sg, s^{\pm}} = m_{i_0} \langle \sg, s^{\pm} \rangle \ge 0, \text{ for any } i_0 \in \I_{j_0}.$$
The result now follows from a similar argument as before. 
\end{proof}

\begin{proof}[Proof of Proposition \ref{prop;coll-free-2}] 
Let $x \in H^1_{t_1, t_2}(x_0, y_0)$ be a fixed-time local minimizer of $\A$. By a contradiction argument, let's assume $coll(x)$ is not empty. Then Lemma \ref{lem;coll-iso} implies the existence of an isolated collision instant $t_0 \in (t_1, t_2)$. Without loss of generality, let's assume $t_0 =0$ and $x(t) \ne 0$, for all $t \in [-\dl, \dl] \setminus \{0\}$ for some $\dl>0$ small. 

By Proposition \ref{prop;Sundman-est}, we can find a sequence of positive numbers $\{\lmd_n \searrow 0\}_{n=1}^{\infty}$, such that  $\lim_{n \to \infty} s(\pm \lmd_{n}) = s^{\pm} \in Cr(U)$. Then $\xb(t)$ defined as \eqref{eq;xbar} is a zero energy homothetic solution. 

Choose an arbitrary $T>0$, by Lemma \ref{lem;homothetic-not-minimizer}, for $\ep>0$ small enough, there is a path $\xb^{\ep} \in H^1_{-T, T}(\xb(-T), \xb(T))$ as defined in \eqref{eq;xb-ep}, such that
\begin{equation*}
 \A(\xb^{\ep}; -T, T) - \A(\xb; -T, T) <0. 
\end{equation*}

For each $\lmd_n$, let $x^{\lmd_n}|_{[-\dl/\lmd_n, \dl/\lmd_n]}$ be the $\lmd_n$ blow-up of $x|_{[-\dl, \dl]}$, and for each $n$ large enough, let $y^{\lmd_n}|_{[-\dl/\lmd_n, \dl/\lmd_n]}$ be defined as
\begin{equation}
\label{eq;y-lmd} y^{\lmd_n}(t) = \begin{cases}
 x^{\lmd_n}(t), \; & \text{ when } t \in [-\dl/\lmd_n, \dl/\lmd_n] \setminus [-T, T]; \\
 x^{\lmd_n}(t) + \phi(t)+ \psi_n(t), \; & \text{ when } t \in [-T, T].
\end{cases}
\end{equation}
with $\psi_n(t)$ as \eqref{eq;psi-n} and 
$$ \phi(t) = \xb^{\ep}(t) - \xb(t), \; t \in [-T, T]. $$
Then Lemma \ref{lem;blow-up} implies 
\begin{equation}
\label{eq;blow-up-lim} \begin{aligned}
\lim_{n \to \infty} & \A(y^{\lmd_n}; -\dl/\lmd_n, \dl/\lmd_n) - \A(x^{\lmd_n};-\dl/\lmd_n,\dl/\lmd_n) \\
& = \lim_{n \to \infty} \A(y^{\lmd_n}; -T, T) - \A(x^{\lmd_n}; -T, T) \\
& = \A(\xb +\phi; -T, T) - \A(\xb; -T, T) = \A(\xb^{\ep}; -T, T) - \A(\xb; -T, T) <0.
\end{aligned}
\end{equation}

Let $y_n(t) = \lmd_n^{\pwa} y^{\lmd_n}(\lmd_n^{-1} t)$, $\forall t \in [-\dl, \dl]$. Then when $n$ is large enough, it is a small deformation of $x(t)$ in a small neighborhood of the origin, and by Lemma \ref{lem;action-homo} and \eqref{eq;blow-up-lim} 
$$ \A(y_n; -\dl, \dl) - \A(x; -\dl, \dl) = \lmd_{n}^{\frac{2-\al}{2+\al}} \left(\A(y^{\lmd_n}; -\dl/\lmd_n, \dl/\lmd_n) - \A(x^{\lmd_n};-\dl/\lmd_n, \dl/\lmd_n) \right) <0,
$$
which is absurd. 
\end{proof}


\section{Bi-hyperbolic solution for Gutzwiller's anisotropic Kepler problem} 

The main purpose of this section is to give a proof of Theorem \ref{thm;bi-hyperbolic-aniso}, so we will only consider the planar case. 
\begin{lem}
\label{lem;HY} Under condition \eqref{U;spiral}, if $x:[T, 0] \to \cc$ is a collision solution with $x(0)=0$ and
\begin{equation}
	\label{eq;x to 1} x(t)/|x(t)| \to \pm 1, \text{ as } t \to 0^+,
\end{equation}
it is not a local minimizer. A similar result holds for an ejection solution $x:[0, T] \to \cc$ with $x(0)=0$. 
\end{lem}

\begin{proof}
In polar coordinates, condition \eqref{U;spiral} is equivalent to 
$$  \tilde{U}''(k\pi) < -\frac{(2-\al)^2}{8}\tilde{U}(k \pi), \; \forall k \in \zz. $$
Meanwhile under the above condition and \eqref{eq;x to 1}, Theorem 1.2 in \cite{HY18} shows the Morse index of the collision solution $x$ is infinite, so it can not be a local minimizer. 

\end{proof}

With this lemma, we can prove the following proposition, which rules out collision in minimizers under certain topological constraints. 
\begin{prop}
\label{prop;coll-free-top-2} Under conditions \eqref{eq;U-aniso} and \eqref{U;spiral}, if $x = re^{i\tht} \in \Lmd_{T^{\pm}}(r_1 e^{i\tht_-}, r_2 e^{i\tht_+})$ satisfies
$$ \A_h(x; T^-, T^+) = \inf \{  \A_h(\xi): \; \forall \xi \in \Lmd(r_1 e^{i\tht_-}, r_2 e^{i \tht_+}) \},$$
then it is collision-free, when $\tht^{\pm}$ satisfy one of the following three conditions:
\begin{enumerate}
\item[(i).] $\tht_- \in [0, \pi]$ and $\tht_+ \in [-\pi, 2\pi];$
\item[(ii).] $\tht_- =0$ and $\tht_+ \in [-2\pi, 2\pi]$;
\item[(iii).] $\tht_- = \pi$ and $\tht_+ \in [-\pi, 3\pi]$. 
\end{enumerate}
\end{prop}

\begin{proof}
By a contradiction argument, let's assume $x$ is not collision-free. Without loss of generality, we may further assume $0 \in (T^-, T^+)$ and $x(0)=0$. The fact that $x$ is an $h$-free-time minimizer under topological constraint implies zero must be the only collision moment, as if there is another collision moment $t_0>0$ (the argument for $t_0 <0$ is the same), by deleting the sub-path $x|_{(0, t_0)}$, we get a new path still belonging to $\Lmd_{T^{\pm}}(r_1 e^{i \tht_-}, r_2 e^{i \tht_+})$, but with strictly smaller $\A_h$ action value, which is absurd. 

For $i=1, \dots, 4$, let $\Pi_i$ be the $i$-th quadrant of the complex plane. Here each $\Pi_i$ should be seen as a close set including its boundary. Notice that under condition \eqref{eq;U-aniso}, the potential $U$ is symmetric with respect to both the real and complex axes. This means we may assume the sub-path $x|_{[T^-, 0]}$ (resp. $x|_{[0, T^+]}$) is entirely contained in the quadrant containing $x(T^-)$ (resp. $x(T^+)$).

We will divide our proof into several different cases.

\emph{Case 1}: $\tht_- \in [0, \pi]$ and $\tht_+ \in [\pi, 2\pi]$. 
By the above argument, we may assume $x([T^-, 0])$ is entirely contained in $\Pi_1$ or $\Pi_2$ depending on the value $\tht_-$. By Proposition \ref{prop;Sundman-est},
$$ \lim_{t \to 0^-} \tht(t) \in \{0, \pi/2, \pi \}.$$
Lemma \ref{lem;HY} then implies the limit must be $\pi/2$, as otherwise $x|_{[T^-, 0]}$ is not a local minimizer, which is absurd. 

Meanwhile $\tht_+ \in [\pi, 2\pi]$ implies $x([0, T^+])$ is entirely contained in $\Pi_3$ or $\Pi_4$ and a similar argument as above shows 
$$ \lim_{t \to 0^{+}} \tht(t) = 3\pi/2. $$ 
Let $\bar{x}: \rr \to \cc$ be the zero energy homothetic solution given as below 
$$ \bar{x} = \begin{cases}
(\kp_+ t)^{\frac{2}{2 +\al}} e^{i \frac{3\pi}{2}}, & \text{ if } t \ge 0; \\
(-\kp_- t)^{\frac{2}{2 +\al}} e^{i \frac{\pi}{2}}, & \text{ if } t \le 0.
\end{cases}
$$
For any $\ep>0$ small enough, define a local deformation of $\bar{x}$ near the origin as 
\begin{equation}
\label{eq;x-bar-ep} \bar{x}^{\ep}(t) = \bar{x}(t) + f_{\ep}(t) \sg, \text{ where } \sg = e^{i \pi}.
\end{equation}
Here $f_{\ep}(t)$ is an even function defined as in \eqref{eq;f-ep}.  

For $\bar{x}$ and $\sg$ given as above, by \emph{Case 2} in the proof of Lemma \ref{lem;homothetic-not-minimizer},
$$ \A(\bar{x}^{\ep}; -T, T) < \A(\bar{x}; -T, T), \;\forall T>0. $$
Meanwhile for a sequence of positive numbers $\{ \lmd_n \searrow 0 \}_{n =1}^{\infty}$, let $x^{\lmd_n}$ be a $\lmd_n$-blow-up of $x$, 
$$ x^{\lmd_n}(t) = \lmd_n^{-\frac{2}{2+\al}} x(\lmd_n t), \;  t \in [T^-/\lmd_n, T^+/\lmd_n]. $$

Fix a $T>0$, for $n$ large enough, we define a deformation of $x^{\lmd_n}$ as 
$$ y^{\lmd_n}(t) = \begin{cases}
x^{\lmd_n}(t), & \text{ when } t \in [T^-/\lmd_n, T^+/\lmd_n] \setminus [-T, T];\\
x^{\lmd_n}(t) + \phi(t) + \psi_n(t), & \text{ when } t \in [-T, T],
\end{cases}
$$
with $\psi_n(t)$ as in \eqref{eq;psi-n} and 
$$ \phi(t) = \bar{x}^{\ep}(t) - \bar{x}(t), \; \text{when } t \in [-T, T]. $$
The above definitions imply $y^{\lmd_n} \in \Lmd_{T^{\pm}/\lmd_n}( \lmd_n^{-\frac{2}{2 +\al}}r_1e^{i \tht_-}, \lmd_n^{-\frac{2}{2+\al}}r_2 e^{i \tht_+} )$. As a result,
$$ y_n(t) = \lmd^{\frac{2}{2+\al}} y^{\lmd_n}(\lmd_n^{-1} t)\Lmd_{T^{\pm}}(r_1 e^{i \tht_-}, r_2 e^{i \tht_+}). $$

Then the same argument as in the proof of Proposition \ref{prop;coll-free-2} implies 
$$ \A(y_n; T^-, T^+) < \A(x; T^-, T^+), $$
which is absurd. 

The proofs for the remaining cases are almost the same, the only difference is to choose the correct $\sg$ in \eqref{eq;x-bar-ep} corresponding to different choices of $\tht_{\pm}$, which are listed below. 

\emph{Case 2}: $\tht_- \in [0, \pi]$ and $\tht_+ \in [0, \pi]$. 
$$ \sg = e^{i \frac{\pi}{2}}. $$

\emph{Case 3}: $\tht_- \in [0, \pi]$ and $\tht_+ \in [-\pi, 0]$.
$$ \sg = e^{i 0}. $$

\emph{Case 4}: $\tht_- = 0$ and $\tht_+ \in [-2\pi, -\pi]$. 
$$ \sg = e^{-i \pi} . $$

\emph{Case 5}: $\tht_- = \pi$ and $\tht_+ \in [2\pi, 3\pi]$
$$ \sg = e^{i 2\pi}. $$

\end{proof}


\begin{proof}[Proof of Theorem \ref{thm;bi-hyperbolic-aniso}]

First let's consider the case that $s^{\pm}$ satisfy condition (i). As the system is symmetric with respect to the real axis, we may further assume $s^-_2 >0$. Condition (i) then implies $s^+_2 \le 0$. As a result, we can find $\tht_{\pm} \in \rr$ satisfying $s^{\pm}=e^{i\tht_{\pm}}$ and
$$  \tht_- \in (0, \pi), \; \tht_+ \in [-\pi, 2\pi] \; \text{ and } \; |\tht_+ -\tht_-| >\pi.
$$

Next let's consider the case that $s^{\pm}$ satisfy condition (ii). As the system is also symmetric with respect to the imaginary axis, we may assume $s^{\pm}=1$.  Then we can find $\tht_{\pm}$, such that 
$$ \tht_-=0, \; \tht_+ \in [-2\pi, 2\pi] \; \text{ and } \; |\tht_+ -\tht_-| >\pi. $$

In both cases, such a pair of $\tht_{\pm}$ must satisfy one the conditions given in Proposition \ref{prop;coll-free-top-2}, then the desired bi-hyperbolic solution can be obtained by the same proof of Theorem \ref{thm;bi-hyper}. 

\end{proof}

\hfill\newline
\noindent{\bf Acknowledgement.} The author thanks Dr. Gian Marco Canneori for a careful reading of the paper and for pointing out several typos in an early draft. The suggestions by the referees are greatly appreciated, especially those that helped improve the statement of Theorem \ref{thm;asym}.

\hfill\newline
\noindent{\bf Data Availability.} Data sharing is not applicable to this article, as no datasets were generated or analyzed during the current study.

\hfill\newline
\noindent{\bf Conflict of interest} The author declares that there is no conflict of interest.

\hfill\newline
\bibliographystyle{abbrv}
\bibliography{refScattering}

\end{document}